\documentclass[11pt]{amsart}
\usepackage[T1]{fontenc}
\usepackage{amsmath}
\usepackage{amssymb}
\usepackage{latexsym}
\usepackage{graphicx}
\usepackage{amsthm}
\usepackage{graphics,epsfig}
\usepackage{color,subcaption}
\usepackage{a4wide,mathrsfs}
\newtheorem{Thm}{Theorem}

\newtheorem{Lem}{Lemma}

\newtheorem{Rem}{Remark}

\title {Monotone smoothing splines with bounds}

\author{Sara Maad Sasane}
\begin{document}
\maketitle \noindent 
\begin{abstract}
The  problem of monotone smoothing splines with bounds is formulated as a constrained minimization problem of the calculus of variations.
Existence and uniqueness of solutions of this problem is proved, as well as the equivalence of it 
 to a finite dimensional but nonlinear optimization problem. A  new algorithm for computing the solution which is a spline curve, 
 using a branch and bound technique, is presented. The method is  applied to examples in neuroscience 
and for fitting cumulative distribution functions from data.

\end{abstract}

\begin{section}{Introduction}\noindent
Splines and smoothing splines have a long history of application in many fields. The basic history is outlined in Egerstedt and Martin, \cite{EM10} and in Wahba, \cite{gW90}. See also \cite{yW11} for an introduction to the use of smoothing splines in statistics. 
In this paper we return to the problem of monotone  smoothing splines, which was previously studied in \cite{EM10,MR945101, MR1707587,NM13}. A classical application is to determine the average growth curve of a population of juveniles. 
Suppose we have a population of perhaps 30 children whose heights are measured every 6 months from age 2 to 20. It is easy to fit a smoothing spline to the data set 
but there is no guarantee of monotonicity. Since people just do not get shorter and then regrow, the smoothing spline is not  appropriate for this and other similar situations. 
Instead, the appropriate tool is the monotone smoothing spline. 
Charles, Sun and Martin \cite{CSM10} used monotone smoothing splines in calculating distribution functions. 
There a problem can arise in that the spline may become greater than $1$ at some point and by monotonicity it can never decrease, and so this violates the condition that a 
cumulative distribution only takes values between $0$ and $1$, a problem that was not addressed in \cite{CSM10}. 
In this paper, we solve this problem by imposing a condition that the spline is bounded above by some number $x_{\max}$
(which would be $1$ in the case of cumulative distribution functions).
We also present an application of monotone smoothing splines in mathematical biology,where sigmoidally shaped function commonly occur.

%This problem has not been addressed 
%in the literature until this paper, and we have examples of how this method can be used to construct cumulative distribution functions using monotone splines.

The main difficulty of the problem is the monotonicity constraint $\dot x\ge 0$, which has to hold at every point $t$ in the interval under consideration. This infinite dimensional
constraint is handled by a vector space version of the Karush--Kuhn--Tucker theorem, and using this, it is possible to reduce the original infinite dimensional problem to a finite
dimensional problem which can be solved numerically. 
%Another problem that can arise in fitting a curve to a growth chart is that the spline may be above the height of the tallest child. Thus it is necessary to impose 
%constraints on the maximum value of the spline at each age. Sometimes it is necessary for the spline to lie in the interval of mean $\pm $ 2 standard deviations at each age. 
%This again is an inequality constraint that must be maintained, which can be handled by a

The paper is organized as follows: In Section 2, we formulate the curve fitting problem as a constrained Calculus of Variations problem, 
and in Section 3, we show existence and uniqueness of the minimizer of this minimization problem, and hence existence and uniqueness of
the monotone smoothing spline function.
In  Section 4, we formulate the Karush--Kuhn--Tucker conditions for this problem, and use these to prove a key lemma, saying that the second derivative of $\ddot x$ is
essentially piecewise linear.
In Section 5, we use the key lemma from Section 4 to reformulate the infinite dimensional linear problem of Section 2 into a finite dimensional but nonlinear problem. This was essentially done previously in 
\cite{EM10}, except that we provide more details in the proof. We also introduce a new branch and bound type
 algorithm for computing the optimal curve. Sections 6 and 7 contain applications and examples which show how the method can be used. In Section 6, we reconstruct 
 sigmoidal shaped curves arising in an intracellular signalling model, while in Section 7, we apply the method on reconstructing cumulative distribution functions using data, and in particular for a cumulative distribution function arising in the cell cycle, and the distribution that we reconstruct gives the time certain cells remain in a particular phase in the cell cycle.

\end{section}
\begin{section}{The problem}\label{S:problem}
Let $T>0$, and let $0=t_0<t_1<\dots<t_m=T$. Consider a data set
$\{(t_i,\alpha_i)\}$ with $\alpha_i\in \mathbb R$, with associated weights $w_i>0$, $i=1,\dots, m$.

Let $x_{\max}>0$, and consider the following optimization problem for functions defined on an interval $[0,T]$:
\begin{equation}\label{E:minproblem}
   \min\left(\frac{1}{2}\int_0^T \ddot{x}(t)^2\, dt + \frac{1}{2}\sum_{i=1}^m w_i(x(t_i)-\alpha_i)^2\right),
\end{equation}
subject to
\begin{equation}\label{E:constraint1}
  \left\{
   \begin{aligned}
      x &\in  H^2((0,T)), \\
      x(0)&=0, \\
      \dot{x}&\ge 0 \text{ on }t\in (0,T), \\
      x(T) &\le x_{\max},
   \end{aligned}\right.
\end{equation}
where $H^2((0,T))$ is the Sobolev space of twice weakly differentiable functions on $(0,T)$.
The condition $x(0)=0$ is included because in many situations in applications, it is for modelling
purposes clear that the curve must satisfy this condition. This happens for example in the application
of the intracellular signalling model that we discuss in Section 6. The condition $x(T)\le x_{max}$ 
arises from $0\le x(t)\le x_{\max}$. Since the curve is monotonically increasing, it suffices to impose
the condition at the endpoint $t=T$. 

As $(0,T)$ is a bounded interval, we can use $\|x\|:= \left(\int_0^T \ddot x(t)^2\, dt\right)^{1/2}$ as a norm on $H^2((0,T))$. 
The rest of the paper is devoted to solving this problem, and to applications of the developed method.
\end{section}
\begin{section}{Existence and uniqueness of a minimizer}
   \begin{Thm}
       Let $X= \{x\in H^2((0,T));\; x(0)=0, \dot x\ge 0, x(T)\le x_{\max}\}$, and assume that $m\ge 1$.
       There exists a unique $x_*\in X$ which solves the minimization problem 
       \begin{equation*}
   \min_{x\in X}\left(\frac{1}{2}\int_0^T \ddot{x}(t)^2\, dt + \frac{1}{2}\sum_{i=1}^m w_i(x(t_i)-\alpha_i)^2\right).
\end{equation*}
   \end{Thm}
   \begin{proof}
      Let $f:X\to \mathbb R$ 
      \begin{equation}\label{E:f-def}
         f(x) = \frac{1}{2}\int_0^T \ddot{x}(t)^2\, dt + \frac{1}{2}\sum_{i=1}^m w_i(x(t_i)-\alpha_i)^2,
      \end{equation}

      We will use the direct method in the calculus of variations, which says that if a functional is coercive on $H^2((0,T))$  and
      weakly lower semicontinuous on a weakly closed set, then a minimizer exists (see e.g. \cite{mS96}, p. 4). 
      It is easy to see that the set $X$ is convex and closed in $H^2((0,T))$ and hence it is weakly closed by Mazur's lemma
      (see e.g. Theorem 3.13 of \cite{wR91}).
      
      We will first check that the first term of $f$ is weakly lower semicontinuous and coercive on $L^2([0,T])$. Indeed, it is weakly
      lower semicontinuous since
      \begin{equation*}
         0\le \int_0^T (\ddot{x}_j(t)-\ddot{x}(t))^2\, dt = \int_0^T \ddot{x}_j(t)^2\, dt -2\int_0^T \ddot{x}_j(t) \ddot{x}(t)\, dt + \int_0^T \ddot{x}(t)^2\, dt,
      \end{equation*}
      and so if $x_j\rightharpoonup x$ (i.e. $x_j$ converges weakly to $x$ in $H^2((0,T))$), then
      \begin{equation*}
         \begin{aligned}
            0&\le \liminf_{j\to\infty} \left(\int_0^T \ddot{x}_j(t)^2\, dt -2\int_0^T \ddot{x}_j(t) \ddot{x}(t)\, dt + \int_0^T \ddot{x}(t)^2\, dt\right) \\
            &= \liminf_{j\to\infty} \left(\int_0^T \ddot{x}_j(t)^2\, dt - \int_0^T \ddot{x}(t)^2\, dt\right),
         \end{aligned}
      \end{equation*}
      i.e.
      \begin{equation*}
         \int_0^T \ddot{x}(t)^2\, dt \le \liminf_{j\to\infty} \int_0^T \ddot{x}_j(t)^2\, dt,
      \end{equation*}
      which shows that $x\mapsto \int_0^T \ddot{x}(t)^2\, dt$ is  weakly lower semicontinuous on $H^2((0,T))$. Coerciveness
      of the first term on $H^2((0,T))$ is obvious since it is the square of the norm on $H^2((0,T))$.

      Weak lower semicontinuity of the second term of \eqref{E:minproblem} follows since $H^2((0,T))$ is compactly embedded in $C^1([0,T])\subset C([0,T])$. Indeed,
      if $x_j\rightharpoonup x$ in $H^2([0,T])$, then $x_j$ is bounded in $H^2([0,T])$, and since the embedding of $H^2((0,T))$ into $C([0,T])$ is compact,
      $x_j$ has a subsequence $x_{j_l}$, which converges (to $x$ by uniqueness of a weak limit) in $C([0,T])$). Finally, note that for each $j=1,\dots,m$, 
      \begin{equation*}
         |x_i(t_j)-x(t_j)|\le \max_{t\in[0,T]} |x_i(t)-x(t)| \to 0
      \end{equation*}
      if $x_i\rightharpoonup x$ in $H^2((0,T))$, and so the sum in \eqref{E:minproblem} is weakly continuous.
     
      As $f$ is a sum of
      two weakly lower semicontinuous functions, it is clear that
      $f$ is weakly lower semicontinuous on $X$. 

      Next, we prove that $f$ is coercive on $X$. As $x\mapsto \int_0^T {\ddot x}(t)^2\, dt$ is coercive on $L^2([0,T])$, we see that
      \begin{equation*}
          f(x) \ge \frac{1}{2} \int_0^T \ddot{x}(t)^2\, dt \to \infty
      \end{equation*}
      as $\|x\|\to \infty$. 
      
      To show uniqueness, we will show that $f$ is strictly convex. For this, we will 
      use that $u\mapsto \int_0^T u(t)^2\, dt$ and $x\mapsto (x-\alpha_1)^2$ are strictly convex 
      on $L^2([0,T])$  and on $\mathbb R$, respectively. 
      Let 
      $f_1(x):=\int_0^T \ddot x(t)^2\, dt$, $f_2(x)=w_1(x(t_1)-\alpha_1)^2$,
      and  $f_3(x)=\sum_{j=2}^m w_j (x(t_j)-\alpha_j)^2$,
      so that $f=f_1+f_2+f_3$ on $X$. It is clear that $f_1$, $f_2$ and $f_3$ are convex
      (but not strictly convex) functions on $X$. 
      To show that $f$ is strictly convex, we need to prove that if
      \begin{equation}\label{E:convexeq}
         f(\lambda x_1 + (1-\lambda) x_2) = \lambda f(x_1) + (1-\lambda) f(x_2)
      \end{equation}
      for some $\lambda\in(0,1)$ and $x_1$, $x_2\in X$, then $x_1=x_2$. 
      
      To prove this, we assume that \eqref{E:convexeq} holds. Since $f_i$ are 
      convex for $i=1,\dots,3$ and $f=f_1+f_2+f_3$, equation \eqref{E:convexeq}
      holds also when $f$ is replaced by $f_i$, $i=1,\dots,3$. Since 
      $u\mapsto \int_0^T u(t)^2\, dt$ is strictly convex, the equality \eqref{E:convexeq}
      for $f_1$ implies that $\ddot x_1= \ddot x_2$. Then, by integration and using that
      $x_1(0)=x_2(0)=0$, it follows that there exists a real constant $A\in \mathbb R$ 
      such that $x_1(t)=x_2(t) +At$. 
      
      On the other hand, 
      since $x\mapsto (x-\alpha_1)^2$ is strictly convex and $w_1>0$,  
      equality \eqref{E:convexeq} for $f_2$ implies that $x_1(t_1)=x_2(t_1)$.
      Combining this with $x_1=x_2+At$, we obtain $A=0$ (since $t_1>0$ 
      and $x_1(0)=x_2(0)$). We have proved that $x_1(t)=x_2(t)$ for all $t\in [0,T]$,
      and hence $x_1=x_2$ in  $X$. This concludes the proof that $f$ is strictly convex
      on $X$, and from this it also follows that the minimizer is unique. 
     \end{proof}
\end{section}

\begin{section}{The Karush--Kuhn--Tucker conditions}
The constrained optimization problem will be solved with a vector space version of the KKT method, cf.
Theorem 1 p. 249 of \cite{dL69}.

Let $X:=\{x\in H^2((0,T));\; x(0)=0\}$, and let $Z:= C([0,T])\times \mathbb R$, with a norm defined by
\begin{equation*}
   \|(u,v_0)\|_Z = \left(\|u\|_{C([0,T])}^2 + |v_0|^2\right)^{1/2}.
\end{equation*}
We
define $G:X\to Z$ by
\begin{equation*}
   G(x) := \left(-\dot{x}, x(T)-x_{\max}\right).
\end{equation*}
We note that since $x\in H^2((0,T))$, it follows that $\dot x\in H^1((0,T))\subset C([0,T])$, and
so it is clear that $G:X\to Z$. 

It is straight-forward to check that $G$ is Fr\'echet differentiable, and its derivative is
\begin{equation*}
   G'(x)(h)  = \left(-\dot{h}, h(T)\right).
\end{equation*}

By the Riesz representation theorem (see e.g. pp. 113--115 of  \cite{dL69}, and pp. 146--150 of \cite{TL80}),
the dual space of $C([0,T])$ is identified with the normalized space of functions of bounded variation on $[0,T]$, denoted by $NBV([0,T])$,
consisting of functions $\nu$ of bounded variation on $[0,T]$ such that $\nu$ is right-continuous and $\nu(0)=0$,
such that the functionals $\phi$ on $C([0,T])$
can be expressed as the Riemann--Stieltjes integral
\begin{equation*}
   \phi(u) = \int_0^T u(t) d\nu(t),
\end{equation*}
and the norm of $\phi$ is the total variation of $\nu$ on $[0,T]$, denoted by $\|\nu\|_{NBV([0,T])}$.

We denote the dual space of $Z$ by $Z^*$. By the above result, $Z^*$ is identified with $NBV([0,T])\times \mathbb R$, and the norm
of an element $(\nu,\mu)\in Z^*$ is given by
\begin{equation*}
   \|(\nu,\mu)\|_{Z^*} = \left( \|\nu\|_{NBV([0,T])}^2 + \mu^2\right)^{1/2}.
\end{equation*}

The positive cone in $Z$ is
\begin{equation*}
   P:=\{(w,\alpha)\in Z;\; w\ge 0 \text{ and }\alpha\ge 0\}.
\end{equation*}
It is clear that $P$ has a nonempty interior. The positive cone $P^*$ in $X^*$ is
\begin{equation*}
 P^*:= \left\{(\nu,\mu)\in NBV([0,T])\times \mathbb R;\; \nu \text{ is nondecreasing and }\mu\ge 0\right\}.
\end{equation*}

   We will derive the KKT conditions for the optimization problem of equation \eqref{E:minproblem}. In order to do this, we first show that
   all points satisfying the inequality $G(x)\le 0$ (i.e. $G(x)\in -P$) are regular points for this inequality
   (cf. \cite{dL69}, p. 248).
   \begin{Lem}
      Every $x\in X$ with $G(x)\le 0$ is a regular point of the inequality $G(x)\le 0$.
   \end{Lem}
   \begin{proof}
      Let $x\in X$ be such that $G(x)\in -P$. We need to show that there exists an $h\in X$
      such that $G(x)+G'(x)h$ is an interior point of $-P$,
      i.e. that $h\in H^2((0,T))$ satisfies
      \begin{equation*}
         \begin{aligned}
            -\dot x -\dot h<0, \\
            x(T)+h(T)-x_{\max}<0.
         \end{aligned}
      \end{equation*}
      There are clearly many choices for $h$, for example
      \begin{equation*}
         h(t) = \frac{x_{\max}}{2T} t - x(t).
      \end{equation*}
      With this choice, we have $h\in X$ and
      \begin{equation*}
         \begin{aligned}
            -\dot x(t) - \dot h(t) &= -\frac{x_{\max}}{2T} < 0, \\ 
            x(T)+h(T)-x_{\max} &= - \frac{x_{\max}}{2} <0,
         \end{aligned}
      \end{equation*}
      i.e. $G(x)+G'(x)h$ is an interior point of $-P$. 
   \end{proof}

The functional
$f:X\to \mathbb R$ defined by \eqref{E:f-def}
is Fr\'echet differentiable, and its derivative is given by
\begin{equation*}
   f'(x)(h) = \int_0^T \ddot{x}(t) {\ddot h}(t)\, dt + \sum_{i=1}^m w_i (x(t_i)-\alpha_i) h(t_i).
\end{equation*}

Let $x_*\in X$ be the minimizer of $f$ subject to $G(x)\in -P$.
By the KKT Theorem (see \cite{dL69}, p. 249), there exists a $z^*\in Z^*$, $z_*\ge 0$ (i.e. $z_*\in P^*$) such that the
Lagrangian
\begin{equation*}
   f(x) + \langle G(x),z_*\rangle
\end{equation*}
is stationary at $x_*$, and that $\langle G(x_*),z_*\rangle =0$.

An explicit statement of the KKT conditions implies the following result, which will be used in the next section
for constructing a numerical algorithm for the
solution:
\begin{Lem}\label{L:piecewise}
   Let $x_*\in X$ be the minimizer of the minimization problem \eqref{E:minproblem} - \eqref{E:constraint1}, and let $u_*:=\ddot{x}_*$. 
   Then the following holds:
   \begin{enumerate}
    \item $u_*$ is affine on each subinterval of $[t_{i-1},t_i)$ where $\dot x_*>0$,
    \item $u_*(0)=u_*(T)=0$.
    \item If $\dot x_*=0$ on an interval, then $u_*=0$ on this interval (and hence it is an affine function also there).  
   \end{enumerate}
\end{Lem}
\begin{Rem} 
   We cannot conclude directly from Lemma \ref{L:piecewise}  that $u_*$ is piecewise linear, since we cannot yet rule out that there is an increasing sequence of points $s_j\to s_0$ 
   such that $\dot x_*(s_j)=0$ while $x_*(t)>0$ for $t\in (s_{j-1},s_j)$ (and $u_*$ is affine on each of the intervals $(s_{j-1},s_j)$). We will see 
   in Lemma \ref{L:optimalinterval},
   that this does not happen for the optimal curve, and $u_*$ is in fact piecewise linear.
\end{Rem}

\begin{proof}
   By the KKT conditions \cite{dL69}, p. 249, there exists a $\nu_*\in NBV([0,T])$ and a $\mu_*\in\mathbb R$ such that
   \begin{equation}\label{E:KKT}
      \begin{aligned}
         \int_0^T \ddot{x}_*(t) {\ddot h}(t)\, dt + \sum_{i=1}^m w_i (x_*(t_i)-\alpha_i) h(t_i) 
         -\int_0^T \dot{h}(t)\, d\nu_*(t) + \mu_* h(T) = 0 
      \end{aligned}
   \end{equation}
   for all $h\in X$.
   Furthermore,
   \begin{equation}\label{E:compslack}
      -\int_0^T \dot{x}_*(t)\, d\nu_*(t) + \mu_* \left(x_*(T)-x_{\max}\right) = 0
   \end{equation}
   where $\nu_*$ is nondecreasing and $\mu_*\ge 0$. Equation \eqref{E:compslack} is the complementary slackness condition, and together with
   the constraint $G(x_*)\le 0$, it implies that $\nu_*(t)$ is constant for $t$ such that $\dot{x}_*(t)> 0$, and that
   $\mu_*=0$ if $x_*(T)-x_{\max}< 0$ .

   The Riemann-Stieltjes integral in \eqref{E:KKT} may be integrated by parts, and doing so and noting that $d{\dot h}(t)=\ddot{h} dt$ (since $\dot h\in H^1((0,T))$ and hence absolutely continuous), 
   we obtain after collecting the two integral terms
   \begin{equation}\label{E:KKT2}
      \begin{aligned}
         \int_0^T (u_*(t)+\nu_*(t)) \ddot{h}(t)\, dt &+ \sum_{i=1}^m w_i\left(x_*(t_i) -
         \alpha_i\right) h(t_i)
         - \nu_*(T)\dot{h}(T)  + \mu_* h(T) = 0,
      \end{aligned}
   \end{equation}
   for all $h\in X$. Choosing $h(t)=0$ on all except one of the subintervals $(t_{i-1},t_i)$, $i=1,\dots,m$, we conclude that for each $i=1,\dots,m+1$, 
   \begin{equation*}
       \int_{t_{i-1}}^{t_i}(u_*(t)+\nu_*(t))\ddot {h}(t)\, dt= 0
   \end{equation*}
   for all $h\in C_0^\infty([t_{i-1},t_i])$. Hence there exist $\beta_i$, $\gamma_i$, $i=1,\dots,m$ such that 
   $u_*(t) + \nu_*(t) = \beta_i+\gamma_i t$ on $(t_{i-1},t_i)$. We may assume (by choosing a representative for the function
   $u_*\in L^2((0,T))$), that $u_*(t)+\nu_*(t)=\beta_i+\gamma_i t$ on the half open interval $[t_{i-1},t_i)$ for $i=1,\dots,m$.        
   Hence $u_*$ is a right continuous function of bounded variation. 
    
    Now with a general $h\in X$, the integral term of \eqref{E:KKT2} can be rewritten using integration by parts as
    \begin{equation*}
       \begin{aligned}
          \sum_{i=1}^{m} \int_{t_{i-1}}^{t_i}(\beta_i + \gamma_i t) \ddot h(t)\, dt &= \sum_{i=1}^{m} \left((\beta_i + \gamma_i t_i)\dot h(t_i) - (\beta_{i}+\gamma_i t_{i-1}) \dot h(t_{i-1}) - \gamma_i(h(t_i)-h(t_{i-1}))\right) \\
          &=  \sum_{i=1}^{m-1} \left((\beta_{i}-\beta_{i+1} + (\gamma_{i} - \gamma_{i+1}) t_{i})\dot h(t_{i})-(\gamma_{i}-\gamma_{i+1})h(t_{i})\right) \\
          &\qquad -\beta_1 \dot h(0)  + (\beta_{m}+\gamma_{m}T)\dot h(T)-\gamma_{m} h(T).
       \end{aligned}      
    \end{equation*}
    By choosing $h$ appropriately (i.e. exactly one of $h(t_i)$ and $\dot h(t_i)$ not equal to zero), we conclude that
    \begin{equation}\label{E:splineeq}
       \begin{aligned}
          \beta_{i}-\beta_{i+1}+(\gamma_{i}-\gamma_{i+1})t_{i} &= 0, \qquad\text{for }i=1,\dots,m-1, \\
          -(\gamma_{i}-\gamma_{i+1})+ w_{i}(x_*(t_{i})-\alpha_{i})&=0, \qquad\text{for }i=1,\dots,m-1, \\
          \beta_1&=0, \\
          \beta_{m} + \gamma_{m} T - \nu_*(T) &= 0, \\
          -\gamma_{m} + \mu_*+w_m(x_*(T)-\alpha_m)&=0,
       \end{aligned}
    \end{equation}
    where we have also used that $h(0)=0$. In particular, since $\nu_*(0)=0$ and $\beta_1=0$, it follows that $u_*(0)=0$. Hence $u_*\in NBV([0,T])$.
    Note that the first equation of \eqref{E:splineeq} implies that $u_*+\nu_*$ is continuous at the spline knots $t_i$, $i=1,\dots,m$, 
    and the second equation implies that the derivative of $u_*+\nu_*$ has jumps of size $-w_{i}(x_*(t_{i})-\alpha_{i})$ at $t_i$, $i=1,\dots,m-1$.

   Recall that $\nu_*$ is (locally) constant for $t$ such that $\dot x_*(t)>0$. Let us examine
   what happens for a point $s_0$ such that $\dot x_*(s_0)=0$.  If $s_0$ is an isolated zero                 
   of $\dot x_*$, then $\nu_*$ may have a jump discontinuity at $s_0$ (where it is right
   continuous). If $\dot x_*=0$ on an interval around $s_0$, then clearly also 
   $u_*=\ddot x_*=0$ on this interval. In particular $u_*$ is piecewise linear in this interval.
\end{proof}

\end{section}

\begin{section}{An algorithm for computing the optimal solution}\label{S:algorithm}
   \noindent Using Lemma \ref{L:piecewise}, we will now reformulate the optimization problem as a finite dimensional
   problem which we can solve numerically. 
   This is essentially the approach of Section 7.3 of \cite{EM10},
   and their method has been adapted to the extra constraint $x(T)\le x_{\max}$. 
   Instead of using dynamical programming as in \cite{EM10}, we give an outline of a branch and bound algorithm 
   for finding an optimal solution to the problem. The variables of this new problem are 
   $x_1,\dots,x_m,v_1,\dots,v_m$, which are the (unknown) values
   of the function $x(t)$ and its derivative $\dot x(t)$  at the spline knots.

   Assuming initially that the values $x_1,\dots,x_n$ and $v_1,\dots,v_m$ are known, we will use the method of \cite{EM10} to determine a
   curve with minimal cost under the constraint that it passes through the points $(t_1,x_1),\dots,(t_m,x_m)$ with derivatives at these points
   equal to $v_1,\dots,v_m$, respectively.
   This will give us a new cost function depending on the variables $x_1,\dots,x_m$ and $v_1,\dots,v_m$. Minimizing this function
   is equivalent to the original infinite dimensional problem.

   Now we focus our attention on one interval $[t_i,t_{i+1}]$, and rename it $[t_0,t_F]$. Without loss of generality,
   we assume that $t_0=0$. The corresponding values of $x$ at the endpoints $0$ and $t_T$ are  denoted by $x_0$ and $x_F$, respectively.
   We assume without loss of generality that $x_0=0$. The values of the derivatives at the endpoints are denoted by
   $\dot x_0$ and $\dot x_F$. The following lemma is essentially given in \cite{EM10}, but here we provide the full proof 
   with more details, taking care of excluding the pathological case in the remark after Lemma \ref{L:piecewise}. 
   Note also that a typo in formula (7.27) of \cite{EM10}  has been corrected
   ($\dot x_i^2$ instead of $\dot x_i$). 
   \begin{Lem}{\cite{EM10}}\label{L:optimalinterval}
      Suppose that $\dot x_0, \dot x_F\ge 0$, $x_F\ge 0$ and $t_F>0$. Then the optimal control
      $u$ which minimizes $\int_0^{t_F} u^2\, dt$ subject to the constraints $\ddot x(t)=u(t)$ for $t\in (0,t_F)$, $x(0)=0$, $x(t_F)=x_F$, $\dot x(0)=\dot x_0$,
      $\dot x(t_F)=\dot x_F$ and $\dot x(t)\ge 0$ for $t\in(0,t_F)$ is given by
      \begin{equation*}
            u(t)= \left(\frac{6(\dot x_0+\dot x_F)}{t_F^2}-12\frac{x_F}{t_F^3}\right)t +
            \frac{6 x_F}{t_F^2}-\frac{4\dot x_0}{t_F}-\frac{2\dot x_F}{t_F}
      \end{equation*}
      if $x_F\ge \frac{t_F}{3}\left(\dot x_0 + \dot x_F - \sqrt{\dot x_0 \dot x_F}\right)$, and
      \begin{equation*}
         u(t) =
         \begin{cases}
            \frac{2\left(\dot x_0^{3/2}+ \dot x_F^{3/2}\right)^2}{9 x_F^2}\left(t - \frac{3 x_F \dot x_0^{1/2}}{\dot x_0^{3/2} + \dot x_F^{3/2}}\right) &\text{if }0\le t<\frac{3x_F \dot x_0^{1/2}}{\dot x_0^{3/2}+\dot x_F^{3/2}}, \\
            0 &\text{if }\frac{3x_F \dot x_0^{1/2}}{\dot x_0^{3/2}+\dot x_F^{3/2}}\le t\le t_F-\frac{3 x_F\dot x_F^{1/2}}{\dot x_0^{3/2} + \dot x_F^{3/2}}, \\
            \frac{2\left(\dot x_0^{3/2}+ \dot x_F^{3/2}\right)^2}{9 x_F^2}\left(t-t_F + 
            \frac{3 x_F \dot x_F^{1/2}}{\dot x_0^{3/2} + \dot x_F^{3/2}}\right) &\text{if }
            t_F-\frac{3 x_F\dot x_F^{1/2}}{\dot x_0^{3/2} + \dot x_F^{3/2}}< t\le t_F
         \end{cases}
      \end{equation*}
      if $x_F < \frac{t_F}{3}\left(\dot x_0 + \dot x_F - \sqrt{\dot x_0 \dot x_F}\right)$.
      The contribution to the cost in the two cases is
      \begin{equation*}
      \int_0^{t_F} u(t)^2\, dt =
      \begin{cases}
         4 \frac{(\dot x_0^2 + \dot x_F^2) t_F^2-3 x_F (\dot x_0 + \dot x_F)t_F + 3 x_F^2 + \dot x_0 \dot x_F t_F^2}{t_F^3} &
         \text{if }x_F\ge \frac{t_F}{3}\left(\dot x_0 + \dot x_F - \sqrt{\dot x_0 \dot x_F}\right), \\
         \frac{2}{9 x_F}\left(\dot x_0^{3/2}+\dot x_F^{3/2}\right)^2 &\text{otherwise.}
      \end{cases}
      \end{equation*}
      The corresponding spline function $x$ is given by
      \begin{equation*}
         x(t) = \dot x_0 t + \left(\frac{3 x_F}{t_F^2} - \frac{2\dot x_0 + \dot x_F}{t_F}\right)t^2 + \left(\frac{\dot x_0+\dot x_F}{t_F^2} 
         - \frac{2 x_F}{t_F^3}\right) t^3
      \end{equation*}
      if $x_F\ge \frac{t_F}{3}\left(\dot x_0 + \dot x_F - \sqrt{\dot x_0 \dot x_F}\right)$, and
      \begin{equation}\label{E:piecewise}
         x(t) = 
         \begin{cases}
             \frac{x_F \dot x_0^{3/2}}{\dot x_0^{3/2} + \dot x_F^{3/2}} + \frac{(\dot x_0^{3/2} + \dot x_F^{3/2})^{2}}{27 x_F^2}\left(t - \frac{3 x_F \dot x_0^{1/2}}{\dot x_0^{3/2} + \dot x_F^{3/2}}\right)^3 & 
             \text{ if }0\le t < \frac{3 x_F \dot x_0^{1/2}}{\dot x_0^{3/2} + \dot x_F^{3/2}} \\
             \frac{x_F \dot x_0^{3/2}}{\dot x_0^{3/2} + \dot x_F^{3/2}} & \text{if }\frac{3 x_F \dot x_0^{1/2}}{\dot x_0^{3/2} + \dot x_F^{3/2}}\le t\le t_F-\frac{3 x_F\dot x_F^{1/2}}{\dot x_0^{3/2} + \dot x_F^{3/2}}, \\
             \frac{x_F \dot x_0^{3/2}}{\dot x_0^{3/2} + \dot x_F^{3/2}} + \frac{\left(\dot x_0^{3/2}+ \dot x_F^{3/2}\right)^2}{27 x_F^2}\left(t-t_F + 
            \frac{3 x_F \dot x_F^{1/2}}{\dot x_0^{3/2} + \dot x_F^{3/2}}\right)^3 &\text{if }
            t_F-\frac{3 x_F\dot x_F^{1/2}}{\dot x_0^{3/2} + \dot x_F^{3/2}}< t\le t_F
         \end{cases}
      \end{equation}
      if $x_F < \frac{t_F}{3}\left(\dot x_0 + \dot x_F - \sqrt{\dot x_0 \dot x_F}\right)$.
   \end{Lem}

   \begin{proof}
      The optimal control $u$ which minimizes $\int_0^{t_F} u^2\, dt$ subject to the constraints $\ddot x=u$, $x(0)=0$, $x(t_F)=x_F$, $\dot x(0)=\dot x_0$, $\dot x(t_F)=\dot x_F$ 
      (that is all the constraints except the monotonicity constraint $\dot x\ge 0$), is an affine function $u(t)=Ct+D$ where $C$ and $D$ are chosen so that the 
      constraints are satisfied. This gives the expression
      \begin{equation*}
          u(t) = \left(\frac{6(\dot x_0+\dot x_F)}{t_F^2}-12\frac{x_F}{t_F^3}\right)t +
            \frac{6 x_F}{t_F^2}-\frac{4\dot x_0}{t_F}-\frac{2\dot x_F}{t_F}
      \end{equation*}
      for all choices of $x_F$, $\dot x_0$, $\dot x_F$ and $t_F$. By integration and using the remaining constraints, the corresponding curve $x(t)$ on the interval $(0,t_F)$ is given by
      \begin{equation*}
         x(t) = \left(\frac{\dot x_0+\dot x_F}{t_F^2} 
         - \frac{2 x_F}{t_F^3}\right) t^3  + \left(\frac{3 x_F}{t_F^2} - \frac{2\dot x_0 + \dot x_F}{t_F}\right)t^2 + \dot x_0 t.
      \end{equation*}
      Clearly, in the cases when the monotonicity constraint $\dot x\ge 0$ is satisfied, this curve is optimal also for the original problem. 
      We claim that the $\dot x\ge 0$ on $(0,t_F)$ if and only if $\frac{x_F}{t_F}\ge \frac{1}{3}\left(\dot x_0 +\dot x_0-\sqrt{\dot x_0 \dot x_F}\right)$. 
      To prove this, note that the quadratic function $\dot x$ is given by
      \begin{equation*}
         \begin{aligned}
         \dot x(t) &= 3 \left(\frac{\dot x_0+\dot x_F}{t_F^2} - \frac{2 x_F}{t_F^3}\right)t^2 +2\left(\frac{3 x_F}{t_F^2}-\frac{2 \dot x_0+\dot x_F}{t_F}\right) t + \dot x_0  \\
                      &=\frac{1}{2} C t^2 + Dt + \dot x_0.
         \end{aligned}
      \end{equation*}
      We note that $\dot x(0)=\dot x_0\ge 0$ and $\dot x(t_F)= \dot x_F\ge 0$ by the assumptions, and so $\dot x(t)\ge 0$ for all $t\in(0,t_F)$ if and only if 
      the value of $\dot x$ at an interior minimizer is nonnegative. 
      
      We study the cases $C\ge 0$ and $C<0$ separately. If $C<0$, then $\dot x$ doesn't have a minimizer, and so the monotonicity condition will always 
      be satisfied. We note that in this case, 
      \begin{equation*}
          \frac{x_F}{t_F}>\frac{1}{2}\left(\dot x_0+\dot x_F\right) \ge \frac{1}{3}\left(\dot x_0 + \dot x_F - \sqrt{\dot x_0 \dot x_F}\right),
      \end{equation*}
      is always satisfied. 
     
      If $C\ge 0$, i e if  $\frac{x_F}{t_F}\le \frac{1}{2}\left(\dot x_0 + \dot x_F\right)$, then we need a necessary and sufficient condition for when the minimum value of $\dot x$ is 
      nonnegative and the minimizer belongs to the 
      interval $(0,t_F)$. The minimizer belongs to the interval if and only if $-\frac{D}{2}\in [0, \frac{t_F C}{2}]$, i e if
      \begin{equation*}
         0\le \frac{2 \dot x_0 + \dot x_F}{t_F} - \frac{3 x_F}{t_F^2} \le 3 t_F\left(\frac{\dot x_0 + \dot x_F}{t_F^2} - \frac{2 x_F}{t_F^3}\right),
      \end{equation*}
      which holds if and only if
      \begin{equation}\label{E:mininint}
         \frac{x_F}{t_F}\le \frac{1}{3}\left(\dot x_0 + \dot x_F + \min(\dot x_0,\dot x_F)\right).
      \end{equation}
      The minimum value $B-\frac{D^2}{2C}$ is nonnegative if and only if $\frac{BC}{2}-\left(\frac {D^2}{2}\right)^2\ge 0$ i e
      \begin{equation*}
         \begin{aligned}
         0&\le 3\dot x_0 \left(\frac{\dot x_0 + \dot x_F}{t_F^2 - 2 \frac{x_F}{t_F^3}}\right) - \left(\frac{3 x_F}{t_F^2} - \frac{2 \dot x_0 + \dot x_F}{t_F}\right)^2 \\
         &= -\left(\frac{3 x_F}{t_F^2}-\frac{\dot x_0 + \dot x_F}{t_F}\right)^2 - \frac{\dot x_0 \dot x_F}{t_F^2},
         \end{aligned}
      \end{equation*}
      and this inequality holds if and only if
      \begin{equation*}
         \frac{1}{3}\left(\dot x_0 + \dot x_F - \sqrt{\dot x_0 \dot x_F}\right) \le \frac{x_F}{t_F}\le \frac{1}{3}\left(\dot x_0 + \dot x_F + \sqrt{\dot x_0 \dot x_F}\right).
      \end{equation*}
      We note that the right inequality is always satisfied if the minimizer belongs to the interval $(0,t_F)$, by \eqref{E:mininint} and since 
      $\min(\dot x_0,\dot x_F)\le\sqrt{\dot x_0 \dot x_F}$.
      To summarize, we see that irrespective of the sign of $C$, a necessary and sufficient condition for the monotonicity of $x$ on $(0,t_F)$ is 
      that 
      \begin{equation*}
         \frac{x_F}{t_F}\ge \frac{1}{3}\left(\dot x_0 +\dot x_F - \sqrt{\dot x_0 \dot x_F}\right),
      \end{equation*}
      as required.
   \end{proof}

   Let
   \begin{equation*}
      V(\Delta x,v_l,v_r,\Delta t) :=
      \begin{cases}
          4\frac{(v_l^2+v_r^2) (\Delta t)^2-3 \Delta x (v_l+v_r) \Delta t+3 (\Delta x)^2+v_l v_r (\Delta t)^2}{(\Delta t)^3} &\text{if }
          \Delta x \ge \frac{\Delta t}{3}(v_l+v_r-\sqrt{v_l v_r}), \\
          \frac{4}{9 \Delta x}\left(v_l^{3/2}+ v_r^{3/2}\right)^2 &\text{otherwise.}
      \end{cases}
   \end{equation*}
   In view of Lemma \ref{L:piecewise} and by considering functions which on
   each subinterval is of the form of Lemma \ref{L:optimalinterval}, we have proved the following:
   \begin{Thm}\label{T:finitedim}
       The infinite dimensional optimization problem of Section \ref{S:problem} is equivalent to the finite
       dimensional problem
       \begin{equation}\label{E:optfun}
          \min \left(\frac{1}{2} \sum_{i=1}^{m-1} V(x_{i+1}-x_i,\dot x_i,\dot x_{i+1},t_{i+1}-t_i) + \frac{1}{2} \sum_{i=1}^m w_i(x_i-\alpha_i)^2\right)
       \end{equation}
       subject to the constraints
       \begin{equation*}
           \begin{aligned}
              -\dot x_i&\le 0\text{ for }i=1,\dots,m, \\
              x_{i}-x_{i+1}&\le 0\text{ for }i=1,\dots,m-1, \\
              x_m&\le x_{\max}, \\
              x_1&=0,
           \end{aligned}
       \end{equation*}
   \end{Thm}
   The nonlinear optimization problem of Theorem \ref{T:finitedim} can be solved numerically, for example with {\tt fmincon} in {\tt matlab}.   
   Unfortunately, this method does not seem to give stable results when there are more than 10 subintervals, and it is hard to analyse due to the piecewise defined 
   objective function. 
   
   To come around this problem, we suggest  using a branch-and-bound approach
   which is outlined below. We emphasise that the algorithm is guaranteed to terminate, since there are finitely many (at most $2^{m+1}$, but probably much less in practice)
   subproblems to solve, each of 
   which are convex and can be solved within a fixed time, e.g. with Newton's method. We suggest that a breadth first search is used when going through the branches of the tree, and it is likely that for most problems it will not be needed to search through so many levels of the tree. Further investigations will be needed to find out how efficient the 
   algorithm is and the limit of the size of the problem that can be solved in practice. These questions will be addressed in a future project. In the current paper (Sections \ref{S:CaMKII} and \ref{S:cumulative}) we give some examples where the method has been implemented for up to 30 data points with good results.
      
    \begin{enumerate}
    \item[1.] Start by fitting an ordinary cubic smoothing spline using the data points and with the constraints that $x(0)=0$ and $x(T)\le x_{\max}$. 
    If this curve satisfies $\dot x(t)\ge 0$ for every $t\in (0,T)$, or, equivalently $x(t_{i+1})-x(t_i)\ge \frac{t_{i+1}-t_i}{3}\left(v_i + v_{i+1} -\sqrt{v_i v_{i+1}}\right)$
    for every $i=0,\dots,m-1$, then this must be the optimal curve, and we can stop. Otherwise, the value of the optimal function for this step gives a lower bound for the
    optimal solution.
    \item[2.] If the curve in step 1 was not optimal, we branch the problem into $m-1$ subproblems, where each subproblem corresponds to an interval for which
    the spline is given by a piecewise defined curve as in \eqref{E:piecewise}, whereas the spline should be an ordinary cubic spline on the other subintervals.
    A minimization problem is solved using \eqref{E:optfun}, except that the first line in the definition is taken for the subintervals where the curve should 
    be an ordinary spline, and the second line for the subinterval where the spline should be piecewise defined. 
    The optimal value for each of these subproblems give lower bounds for the optimum of that branch. If the curve is monotone, then we have an optimum for the current
    branch and don't need to branch any further. Otherwise, that node have to be branched into further subproblems, each with one more subinterval where we use a piecewise
    defined spline curve.
    \item[3.] We continue branching and bounding.
    Branches whose lower bound is smaller than an optimum in another side branch can be cut off, and don't need to be examined further.
    In the end, we compare the branch with the smallest optimum, which will give the minimum of the full problem.
   \end{enumerate}

   In Sections \ref{S:CaMKII} and \ref{S:cumulative}, we show how this method can be  used to construct graphs of increasing curves relevant in applications from computational biology and for finding cumulative distribution functions.

   \end{section}

   \begin{section}{Applications to an intracellular signaling model}\label{S:CaMKII}
   \begin{figure}[ht]
      \centering
      \begin{subfigure}[b]{0.33\textwidth}
         \includegraphics[width=\textwidth]{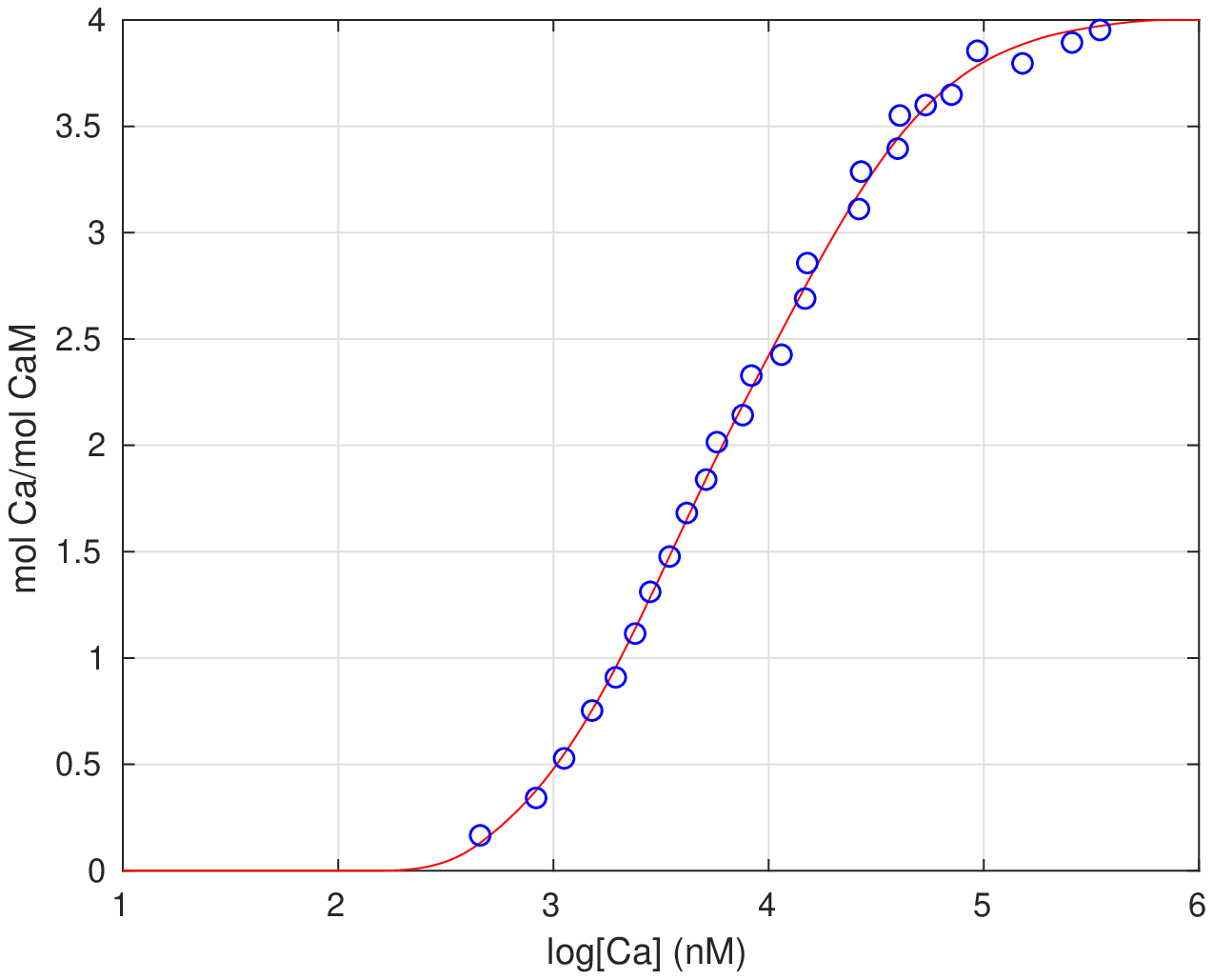}
      \end{subfigure}%
      \begin{subfigure}[b]{0.33\textwidth}
         \includegraphics[width=\textwidth]{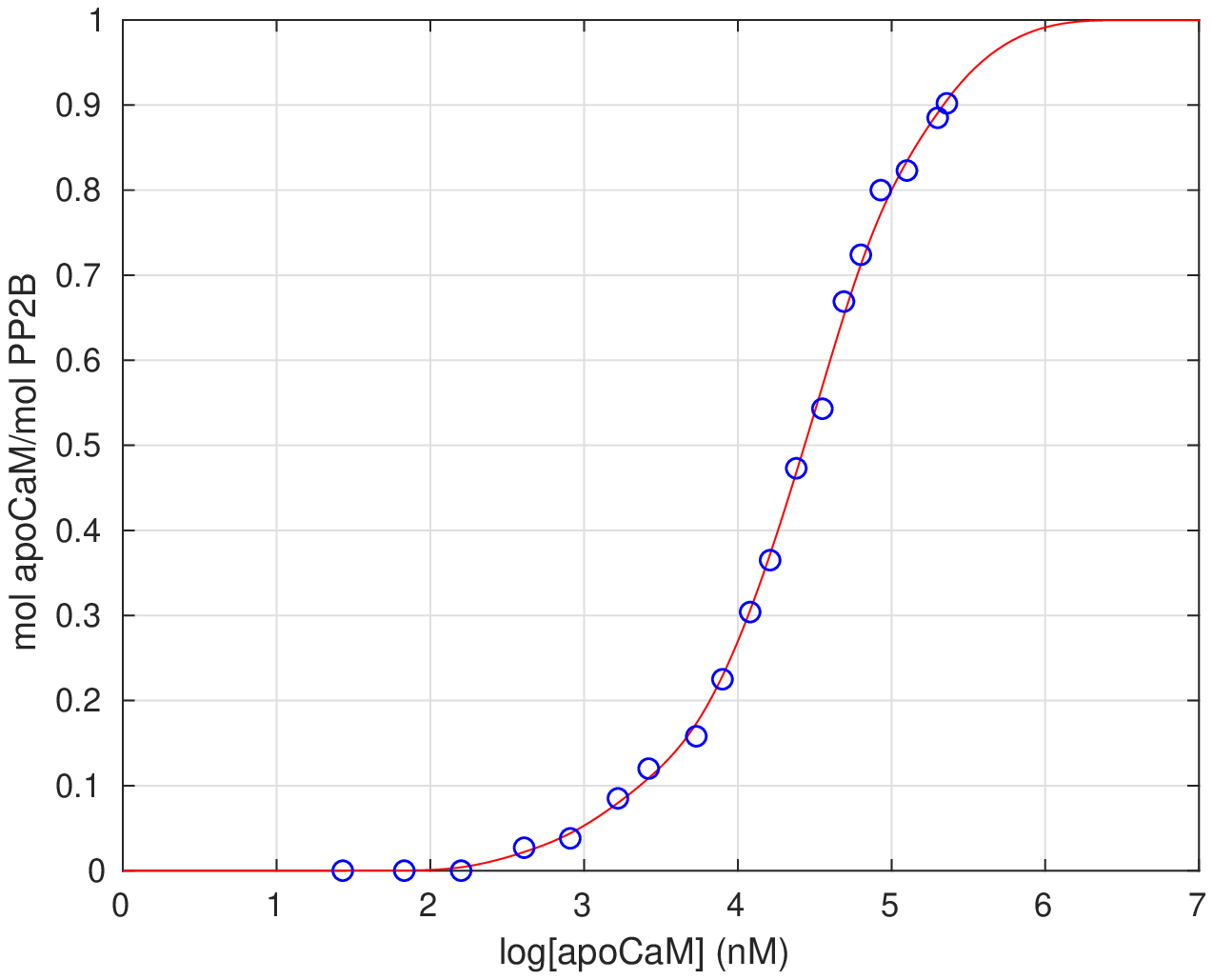}
      \end{subfigure}%
      \begin{subfigure}[b]{0.33\textwidth}
         \includegraphics[width=\textwidth]{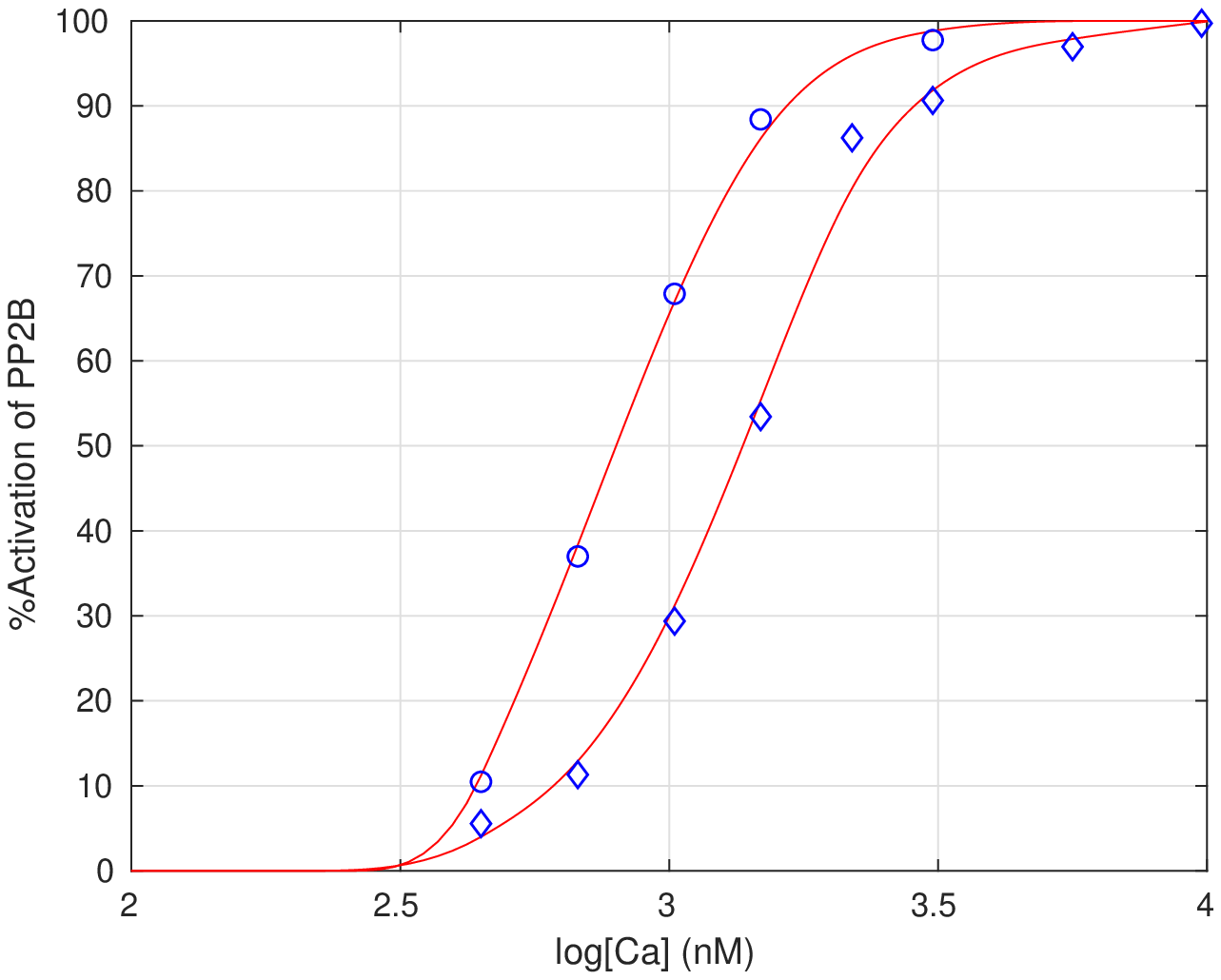}
      \end{subfigure}
      \begin{subfigure}[b]{0.33\textwidth}
          \includegraphics[width=\textwidth]{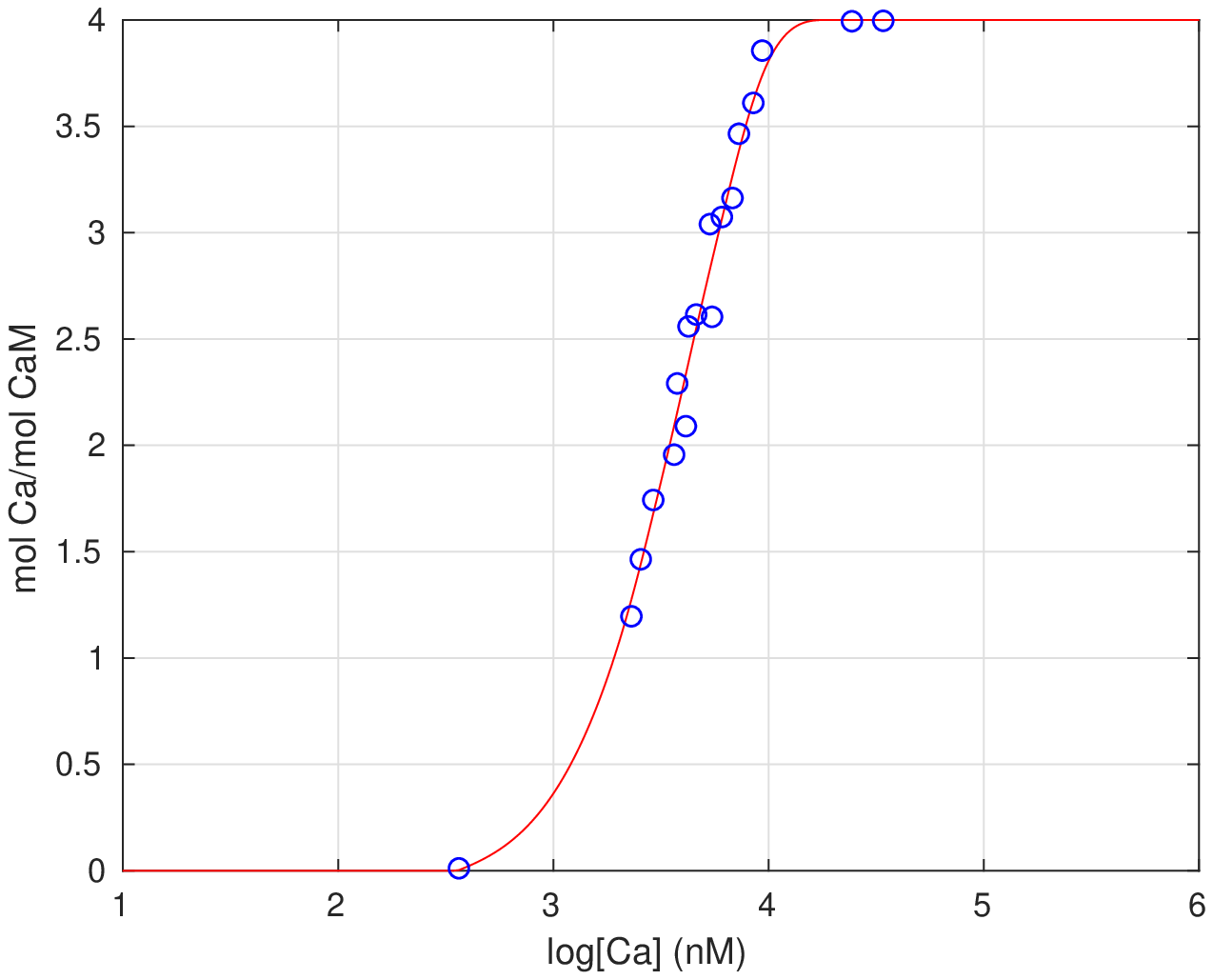}
      \end{subfigure}%
      \begin{subfigure}[b]{0.33\textwidth}
         \includegraphics[width=\textwidth]{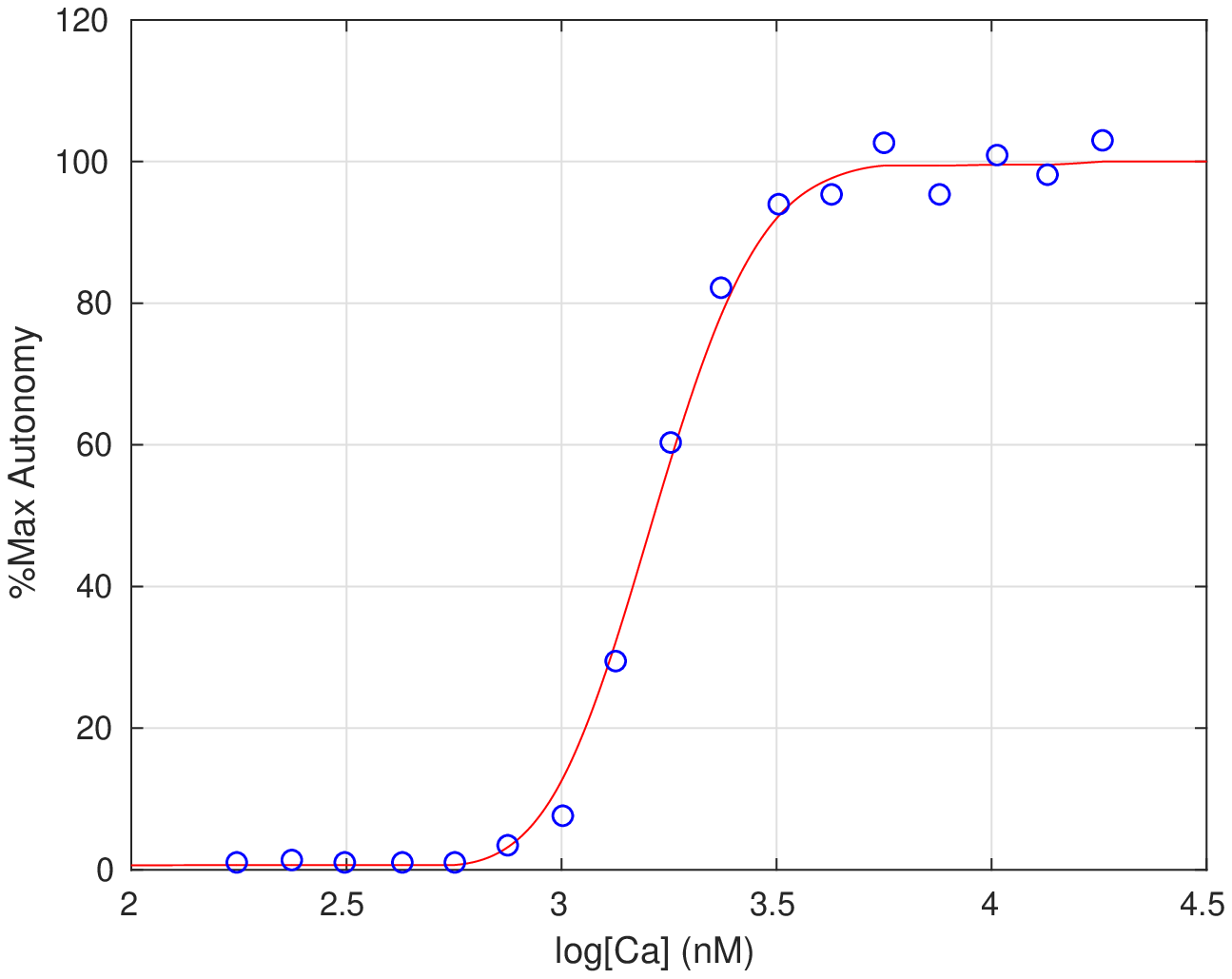}
      \end{subfigure}
      \caption{Reconstruction of the first five plots of Figure 12.3 of \cite{NGEJBK14} with monotone smoothing splines, using the branch and bound algorithm
      outlined in Section \ref{S:algorithm}. The data for the respective subfigures are taken (in order) from \cite{SK94, OYFS11, SK94, Shifman13968, Bradshaw10512}.}
      \label{F:CaMKII}
   \end{figure}
   \noindent 
   As a first application of the algorithm developed in Section 
   \ref{S:algorithm}, we suggest curve fitting using monotone smoothing splines
   as an alternative to the parametric models that are commonly used in modelling
   of pathways of the cell. This is expected to be particularly useful in cases where the underlying
   chemistry is not completely understood, but when certain monotonicity trends in the
   data can be observed.
   The ODE models or stochastic models that are commonly used can be very complicated, see e.g. \cite{NGEJBK14}
   where the modelling process of this type of models is described. In Section 3 of this reference, a relatively simple modelling example of this type, occuring in    neuroscience is given, which we describe briefly here, to give the reader context. 
  
   Calmodulin is an abbreviation for calcium-modulated protein, which is an intermediate calcium-binding 
   messenger protein present in all eukaryotic cells.  Once bound to a calcium ion, calmodulin acts as part of a calcium signal pathway by modifying its 
   interactions with various target proteins such as kinases or phosphatases. Its importance in neuroscience stems from its crucial involvement in
   synaptic plasticity. 
      
   We consider five data sets with experimental data taken from \cite{ Bradshaw10512, OYFS11, Shifman13968, SK94}, 
   corresponding to models describing this particular pathway. See \cite{NGEJBK14, EJMKNSH19} for a model using a system
   of ordinary differential equations stemming from steady state equations for these reactions, and where the same data sets are used. 
   This particular model consists of the elementary species calcium (Ca), calmodulin (CaM), protein phosphatase 2B (PP2B), and Ca/CaM-dependent protein kinase II (CaMKII) and protein phosphatase 1 (PP1).
   
   Up to four Ca ions are bound by calmodulin,
   and the first data set that we consider describes how many (moles of) ions of calcium are bound to 
   CaM per (mole of) CaM, and this quantity is plotted versus the 
   Ca concentration. When the Ca concentration
   increases, more Calcium ions are bound to the proteins, and it is therefore 
   natural to assume that the curve corresponds to the graph of an increasing function. 
   Each calmodulin molecule can bind at most four Ca ions, and hence the range of the function is 
   naturally included in $[0,4]$. When there is no Ca present in the system, there cannot be any
   bounds, and so we require that the curve passes through the origin. The data set is taken 
   from \cite{SK94}, and the data together with the fitted monotone spline is shown in the first 
   subfigure of Figure \ref{F:CaMKII}. 
   
   The binding of Ca ions by Calmodulin is a cooperative process. Ca-bound CaM activates PP2B, another protein implicated in molecular processes related to learning which also plays a role in striatal signaling. 
   Dataset 2, which is taken from \cite{OYFS11}, describes the number of
   moles of apo calmodulin (apoCaM, i e calmodulin without calcium bound to it) bound to each mole PP2B versus the concentration of apoCaM.
   Since one PP2B molecule can bind at most one calmodulin molecule, the 
   range of the function is naturally between $0$ and $1$. As there has to be apoCaM in the system for this type of binding to occur, we require that the $(0,0)$ is on the curve. The more apoCaM there is in the system, the more likely it is for such a binding to occur, and it is hence natural to assume that the fitted function is increasing. 
   
   In the third subfigure, percentage activation of PP2B is plotted versus Ca
   concentration at two different concentrations of CaM (30 nM for the left curve and 300 nM for the right curve). The data for this subfigure is taken from \cite{SK94}. Naturally, the range of the function is contained in $[0,100]$, and the data suggests that the binding is more likely to occur for higher concentrations of Ca, and hence it is natural to fit the data with a curve
   which is the graph of an increasing function. Again, Ca is needed for the activation to occur, and for this reason we require that the curve
   starts at the origin.
   
   The third protein CaMKII, is a kinase, which is activated by the binding of Ca--CaM. In the fourth subfigure, we consider data from \cite{Shifman13968}, representing the number of moles of Ca that is bound to CaM per mole of CaM in the presence of the enzyme CaMKII. As in subfigure 1, the curve is expected to be the graph of a function which is increasing, whose range is contained in $[0,4]$, and which is originating from the origin. 
   
   CaMKII molecules exist as dodecamers, consisting of two hexamer rings. 
   A CaMKII unit that has bound CaM can autophosphorylate when sitting beside an active neighboring unit in the same hexamer ring. The phosphorylated unit can remain active even in the absence of Ca-CaM. 
   In subfigure 5, the data comes from \cite{Bradshaw10512}, and it describes
   the percentage phosphorylated CaMKII (autonomy in CaMKII activity) versus calcium concentration.
   
   The method of the current paper is applied to  data sets in order to fit curves which
   come close to the data points. The weights $w_i$ were set to $1$
   for all examples and the parameter $\lambda$ was chosen to be $100$ 
   for the first two datasets and $1000$ for the three last.
   
   The same range was used for the variables as in Figure 12.3 of 
   \cite{NGEJBK14}. Instead of imposing a new type of condition corresponding
   to the limit as the independent variable tends to infinity, an additional data point was introduced,
   which forces the curve to come close to the maximum
   bound at the right endpoint of the interval. For example, for dataset 1,
   we demand that the curve comes close to the point $(6,4)$. 
   After doing this, the method of Section \ref{S:algorithm} could be used directly. The 
   plots of Figure \ref{F:CaMKII} were obtained, and these can be compared to the first five plots in 
   Figure 12.3 of \cite{NGEJBK14}. 
   
      \end{section}
   \begin{section}{Applications for cumulative distribution functions and an example from cell cycle models}\label{S:cumulative}
      A second application to the technique of this paper, is for reconstructing an unknown distribution function given
      some data points.
      
      Suppose that it is known that a sample comes from a distribution with an absolutely continuous distribution function,
      but that the exact form of the distribution is unknown.
      Then we could reconstruct the cumulative distribution function by using monotone splines. 
      We first test the method from a sample of data points coming from a normal distribution.
      
      Using $x_{max}=1$, $1000$ random points were generated following the normal distribution with expectation value $0$ and
      standard deviation $1$. Then a histogram with $20$ bins was created using Matlab's function {\tt histcounts}. The vector
      $\alpha$ with data values was created by using Matlab's function {\tt cumsum}. With $\lambda=50$, the method 
      of this paper was used with the minor modification that $u(0)=0$ is replaced by $u(t_0)\ge 0$ to create an approximation of the 
      cumulative distribution function. By differentiating the obtained spline function
       an estimate for the density function could also be obtained.
      The results can be seen in Figure \ref{F:normal}.
      \begin{figure}[ht]
          \centering
      \begin{subfigure}[b]{0.44\textwidth}
         \includegraphics[width=\textwidth]{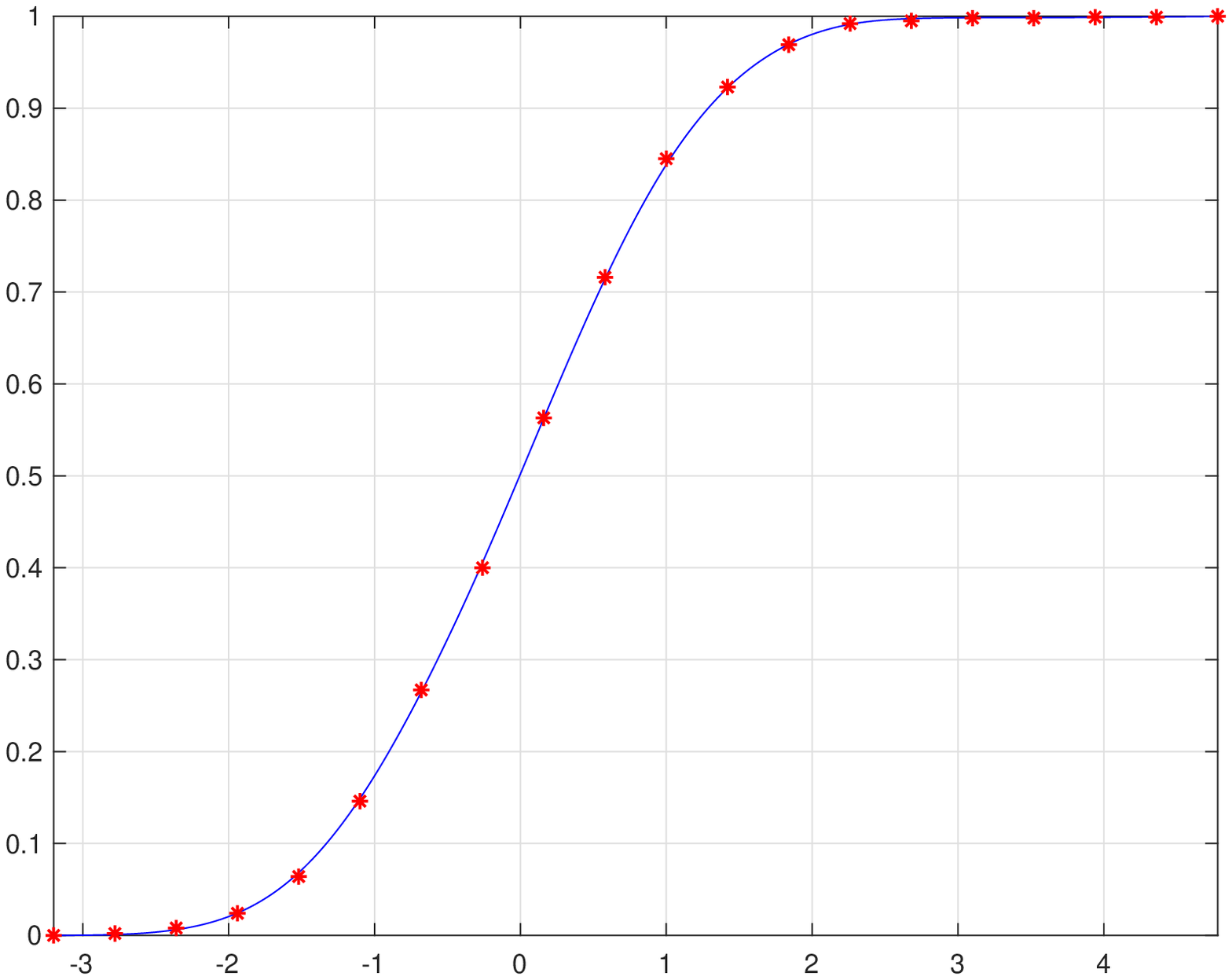}
      \end{subfigure}%
      \begin{subfigure}[b]{0.56\textwidth}
         \includegraphics[width=\textwidth]{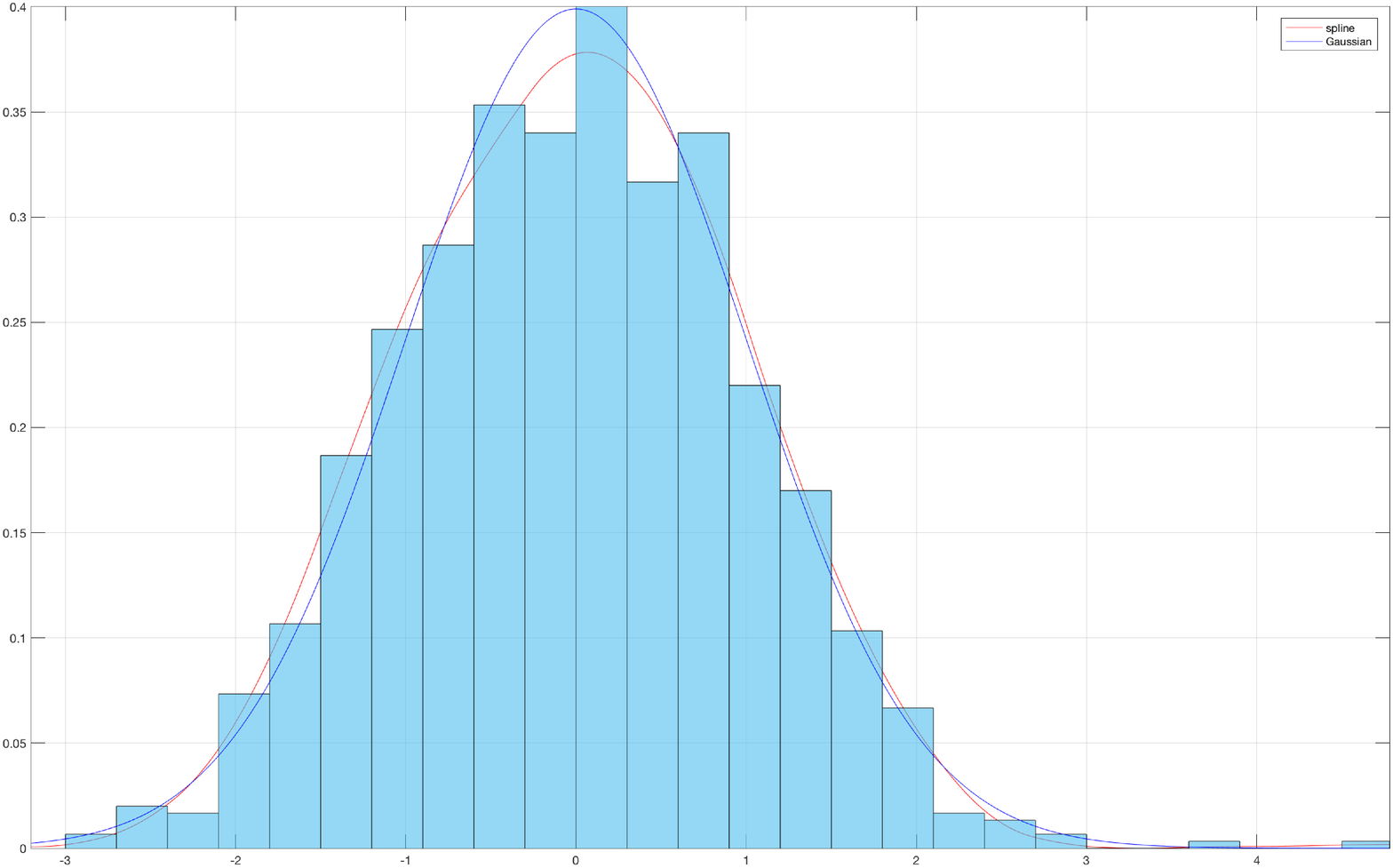}
      \end{subfigure}%      
      \caption{Estimate of the cumulative distribution function using monotone splines (left), and comparison of the derived density function with the exact density function of the normal distribution and the histogram (right).
\label{F:normal}}
      \end{figure}
     
      The method is more useful when the data does not come from a standard distribution, and the following is an example of such a situation arising in cell biology. 
      The cell cycle consists of four distinct phases: $G1$, $S$, $G2$ and $M$. For many types of cells, the time a cell spends in the $G1$ phase is highly variable, and it 
      is of interest to find the distribution for the time a cell spends in the $G1$ phase. See for example \cite{sMS16}, where such a distribution is used in an age structured
      cell cycle model. 
      FUCCI is a fluorescence technology that can be used for tracking the time an individual 
      cell spends in the $G1$ phase \cite{SKMHHOKFMMIOMM08, SOHKOM08}. Using the movie S1 of the supplementary material of \cite{SKMHHOKFMMIOMM08}, the histogram data
      for the time that each cell in that movie stays in the $G1$ phase was obtained. Using $\lambda=3$,  the cumulative distribution function could be estimated and is shown in Figure \ref{F:FUCCI}
       \begin{figure}[h]
          \centering
      \begin{subfigure}[b]{0.5\textwidth}
         \includegraphics[width=\textwidth]{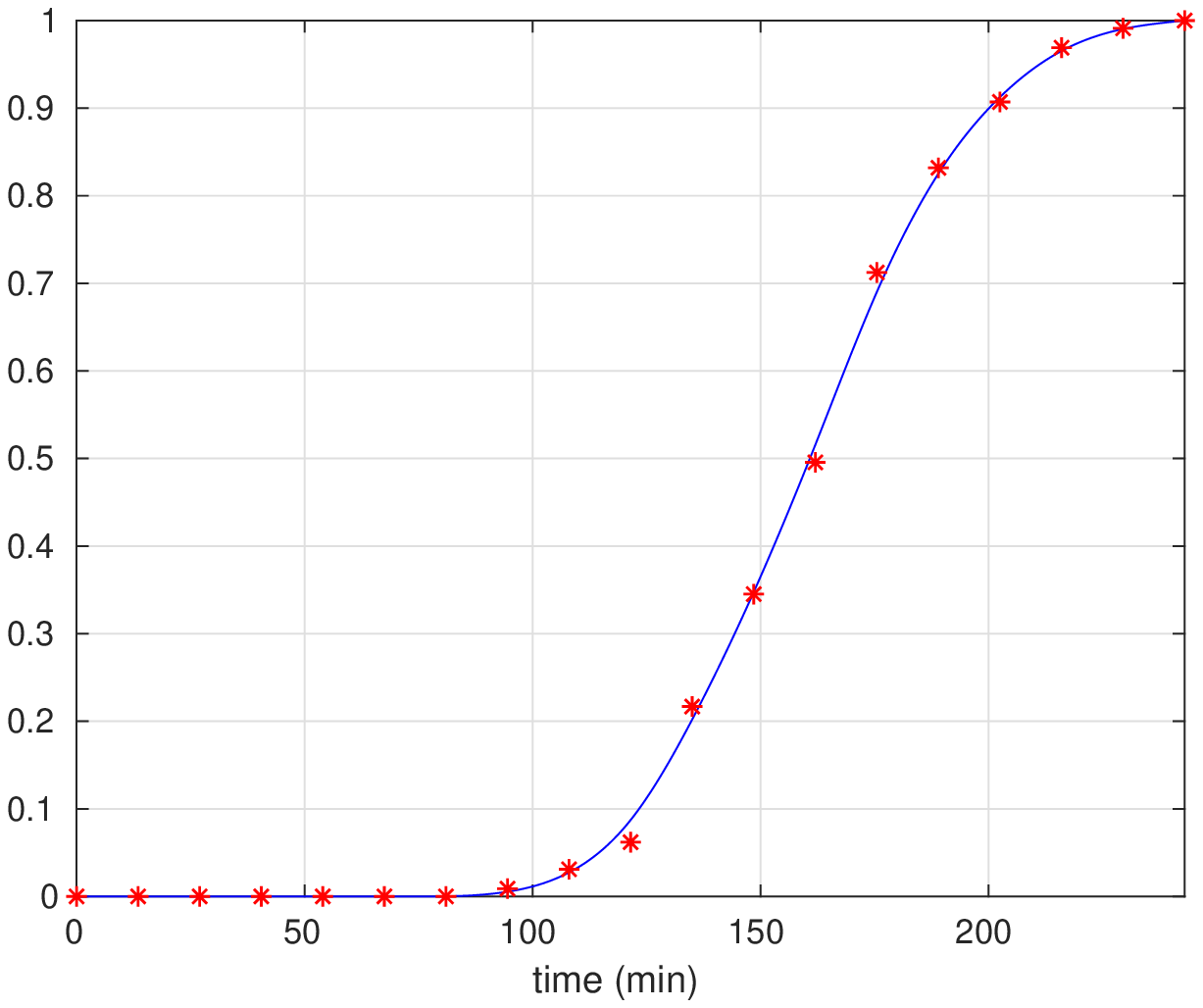}
      \end{subfigure}%
      \begin{subfigure}[b]{0.5\textwidth}
         \includegraphics[width=\textwidth]{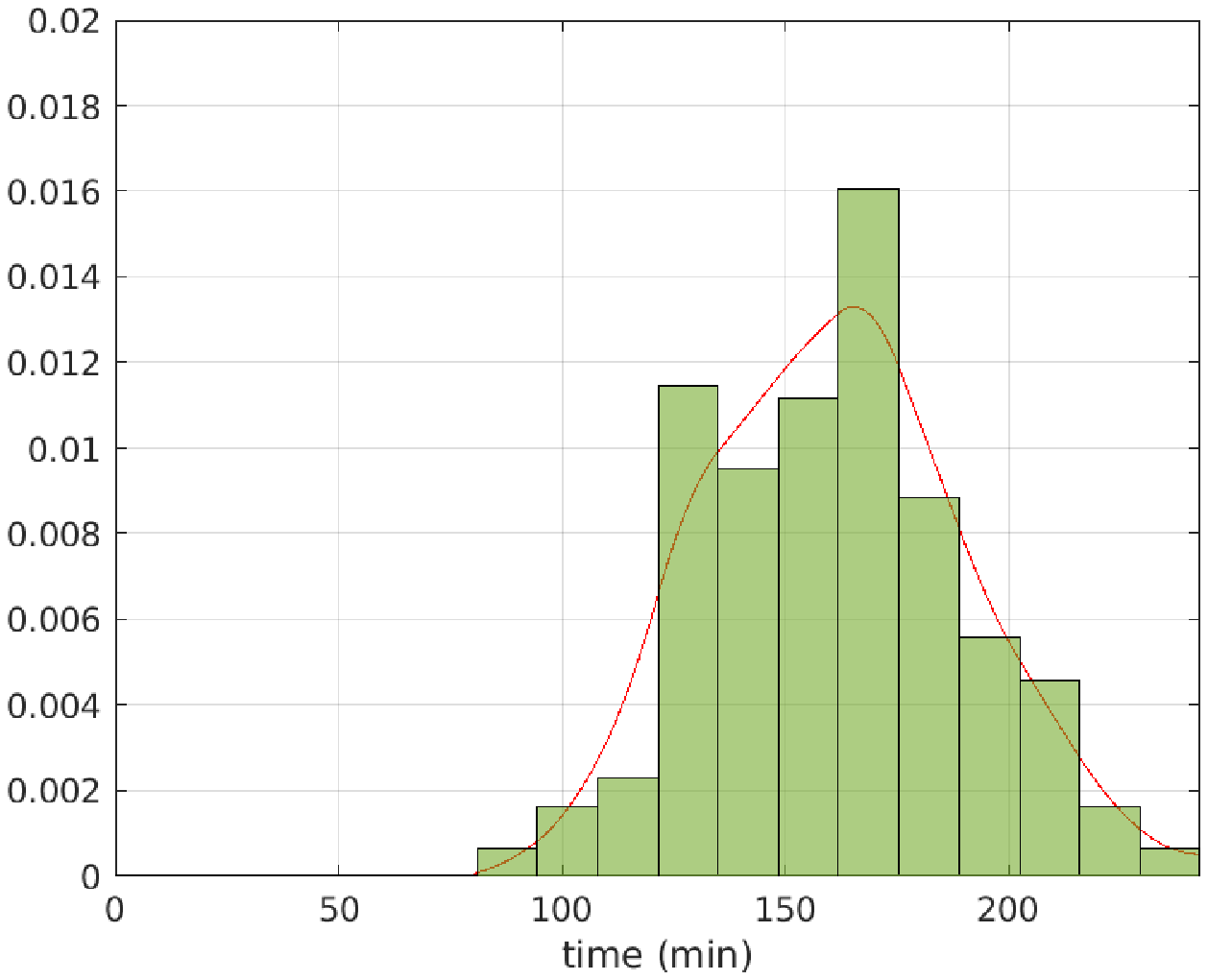}
      \end{subfigure}%      
      \caption{Estimate of the cumulative distribution function using monotone splines (left), and comparison of the derived density function with the histogram (right).
\label{F:FUCCI}}
      \end{figure}
      
   \end{section}
\begin{section}{Acknowledgements}
The author would like to thank Clyde Martin for reading and commenting on an earlier version of the manuscript and
 Olivia Eriksson for a useful discussion about the data sets occurring in Figure \ref{F:CaMKII}. She is also thankful to the reviewers
 for useful comments which led to an improvement of the paper.
\end{section}

\bibliography{referens}
\bibliographystyle{acm}
\end{document}